\documentclass[11pt]{article}

\usepackage{amsfonts,amsmath,amsthm,latexsym,graphicx,hyperref,color,epsfig,mathrsfs,enumerate}

\newtheorem{thm}{Theorem}[section]

\newtheorem{lem}[thm]{Lemma}

\newcommand{\A}{\mathcal{A}}

\title{Forbidding intersection patterns between layers of the cube}
\author{Eoin Long\thanks{School of Mathematical Sciences, Queen Mary University of London, Mile End Road, London E1 4NS, United Kingdom. E-mail: eoinlong@post.tau.ac.il}}
\date{}

\begin{document}

\maketitle

\begin{abstract}
 A family ${\mathcal A} \subset {\mathcal P}[n]$ is said to be an antichain if $A \not \subset B$ for all distinct $A,B \in {\mathcal A}$. A classic result of Sperner shows that such families satisfy $|{\mathcal A}| \leq \binom {n}{\lfloor n/2\rfloor }$, which is easily seen to be best possible. One can view the antichain condition as a restriction on the intersection sizes between sets in different layers of ${\mathcal P}[n]$. More generally one can ask, given a collection of intersection restrictions between the layers, how large can families respecting these restrictions be? Answering a question of Kalai \cite{kal}, we show that for most collections of such restrictions, layered families are asymptotically largest. This extends results of Leader and the author from \cite{LL}.
\end{abstract}

%
%
% next page
%
%
\section{Introduction}

A family ${\mathcal A} \subset {\mathcal P}[n]$ is said to be an antichain if $A \not \subset B$ for all distinct $A,B \in {\mathcal A}$. A classic result in extremal combinatorics is Sperner's theorem \cite{sper}, which shows that any such family ${\mathcal A}$ has size at most $\binom {n}{\lfloor n/2\rfloor }$. This is easily seen to be best possible. This result has been hugely influential, having numerous interesting applications and extensions (for example, see \cite{com} and \cite{eng} for an overview of some of these directions).

Recently, Sperner's theorem was applied in a new proof of Furstenberg and Katznelson's density Hales-Jewett theorem by the polymath internet project (\cite{DHJ}, \cite{FK}). Here, roughly speaking, Sperner's theorem (and a multi-dimensional extension of Gunderson, R\"odl and Sidorenko \cite{GRS}) form a base level of an induction hypothesis. While weaker than Sperner's theorem, a crucial fact here was that any Sperner family ${\mathcal A}\subset {\mathcal P}[n]$ satisfies $|{\mathcal A}|= o(2^n)$. 

Motivated by its place in the proof of the density Hales-Jewett theorem, Kalai \cite{kal} asked whether it is possible to obtain similar results for other `Sperner-like conditions'. One example of such a condition was the tilted Sperner condition considered in \cite{LL}. Kalai noted that the Sperner condition can be rephrased as follows: ${\mathcal A}$ does not contain two sets $A$ and $B$ such that, in the unique subcube of ${\mathcal P}[n]$ spanned by $A$ and $B$, $A$ is the bottom point and $B$ is the top point. He asked: what happens if we forbid $A$ and $B$ to be at a different position in this subcube? In particular, he asked how large ${\mathcal A} \subset {\mathcal P}[n]$ can be if we forbid $A$ and $B$ to be at a `fixed ratio' $p:q$ in this 
subcube. That is, we forbid $A$ to be $p/(p+q)$ of the way up this subcube and $B$ to be $q/(p+q)$ of the way up this subcube. Equivalently, $q|A\setminus B| \neq p|B\setminus A|$ for all distinct $A,B\in {\mathcal A}$. Note that the Sperner condition corresponds to taking $p=0$ and $q=1$. In \cite{LL}, an asymptotically tight answer was given for all ratios $p:q$, showing that one cannot improve on the `obvious' example, namely the $q-p$ middle layers of ${\mathcal P}[n]$. 
\begin{thm}[\cite{LL}]
 \label{tiltedsperner}
  Let $p,q$ be coprime natural numbers with $q\geq p$. Suppose ${\mathcal A} \subset {\mathcal P}[n]$ does not contain distinct $A,B$ with
  $q|A\setminus B| = p|B\setminus A|$. Then 
\begin{equation}
   |\A| \leq (q-p + o(1))\binom {n}{n/2}.
\end{equation}
\end{thm}
Up to the $o(1)$ term, this is best possible. Indeed, the proof of Theorem \ref{tiltedsperner} in \cite{LL} also gives the exact maximum size of such ${\mathcal A}$ for infinitely many values of $n$.

Here we will view the Sperner condition from a slightly different perspective. Given $i \in [0,n]$, let $[n]^{(i)} = \{A \subset \{1,\ldots ,n\}: |A| = i\}$ and given a family of sets ${\mathcal A} \subset {\mathcal P}[n]$, let ${\mathcal A}^{(i)}$ denote the set ${\mathcal A}^{(i)} = \{A\in {\mathcal A}:|A| = i\}$.

\begin{definition}
A family ${\mathcal A} \subset {\mathcal P}[n]$ satisfies an $x_{ij}$-\emph{pairwise restriction} between layers $i$ and $j$ of the cube if $|A\setminus B| \neq x_{ij}$ for all $A\in {\mathcal A}^{(i)}$ and $B \in {\mathcal A}^{(j)}$. 
\end{definition}
\vspace{.2cm}

Both the Sperner and tilted Sperner conditions can be viewed as collections of pairwise restrictions between layers of the cube. Indeed, $\mathcal A$ is a Sperner family if and only if $ |A\setminus B|\neq 0$ for all $A\in {\mathcal A}^{(i)}$ and $B\in {\mathcal A}^{(j)}$ whenever $i<j$. Similarly the tilted Sperner conditions can be viewed as a collection of pairwise restrictions; for example, a small calculation shows that $\mathcal A$ is a 1:2-tilted Sperner family if and only if $|A\setminus B|\neq j-i $ for all $A\in {\mathcal A}^{(i)}$ and $B\in {\mathcal A}^{(j)}$ for some pairs $\{i,j\}$ (those $i<j$ which satisfy $j\leq 2i$ and $2j-i\leq n$). 
The main question we consider in this paper is the following: given a collection of pairwise restrictions between layers of the cube, how large can families respecting these restrictions be?

We represent a collection of pairwise restrictions by a pair $(G,\mathbf{x})$, where $G$ is a graph with vertex set $\{0,\ldots ,n\}$ and $\mathbf{x}=(x_{ij})$ is a vector whose coordinates are indexed by the edges of $G$. An edge $ij$ of $G$ indicates that there is a pairwise restriction between 
sets in $[n]^{(i)}$ and those in $[n]^{(j)}$. The entry $x_{ij}$ of $\mathbf{x}$ corresponding to this edge $ij$ then tells us what this restriction is:
\begin{equation*}
 |A\setminus B|\neq x_{ij}
\end{equation*} 
for sets $A\in [n]^{(i)}$ and $B\in [n]^{(j)}$. Note that since $|A\setminus B| \leq \min (|A|, |B^c|)$, this condition is vacuous unless $x_{ij}\in [0,\min (i,n-j)]$. 

%\begin{figure}
%\centering
%\input{graphtopowerset.pstex_t}
%\caption{An edge $ij$ of $G$ with label $x_{ij}$ indicates the pairwise %restriction $|A\setminus B| \neq x_{ij}$ between 
%sets $A \in [n]^{(i)}$ and $B \in [n]^{(j)}$.}
%\end{figure}

% C4 S1 P5
% Definition of (G,x)-Sperner family.

\begin{definition}
Let $G$ be a graph on $\{0,\ldots n\}$ and let $x_{ij}\in [0,\min (i,n-j)]$ for all 
$ij\in E(G)$ with $i<j$. A family ${\mathcal A}\subset {\mathcal P}[n]$ is a $(G,\mathbf{x})$-Sperner 
family if for every edge $ij\in E(G)$, $|A\setminus B|\neq x_{ij}$ for all sets 
$A\in {\mathcal A}^{(i)}$ and $B\in {\mathcal A}^{(j)}$.
\end{definition}
\vspace{.2cm}

% C4 S1 P6
% Relation to Sperner and tilted Sperner families

We will be mainly concerned with the cases where $i,j \approx n/2$ which gives $\min (i,n-j) \approx n/2$, as by Chernoff's inequality (\cite{chernoff}) most elements of ${\mathcal P}[n]$ 
lie in this range. 

In this language a Sperner family is 
just a $(K_{n+1},\mathbf{0})$-Sperner family. Similarly, a 1:2-tilted Sperner family is a 
$(G ,\mathbf {x})$-Sperner family where $ij\in E(G) \Leftrightarrow 2i\geq j \mbox{ and } 2j-i\leq n$
and $x_{ij} =j-i$ for all edges $ij\in E(G)$. Our main question can now be rephrased as 
follows: given $G$ and $\mathbf {x}$, how large can the $(G,{\mathbf{x}})$-Sperner families be?

% C4 S1 P7
% Construction of large Sperner families

One easy way to construct a large $(G,\mathbf{x})$-Sperner family is to take ${\mathcal A}$ to 
be a union of layers with no pairwise restrictions between them. Equivalently, 
\begin{equation}
 \label{GxSpernerconstruction}
 \mathcal A = \bigcup _{i\in I} [n]^{(i)}
\end{equation}
for an independent set $I$ of $G$. This shows that we can always find a $(G,\mathbf{x})$-Sperner 
family of size at least 
\begin{equation*}
w(G)=\max _{I} \sum _{i\in I} \binom {n}{i}
\end{equation*}
where here the maximum is taken over all independent sets $I$ in $G$. We call $w(G)$ the \emph{weight} 
of $G$. Furthermore, for the Sperner and tilted Sperner conditions, $w(G)$ actually gives the maximal size of Sperner 
and tilted Sperner families. Indeed, $G = K_{n+1}$ for the Sperner condition so $w(G) = \binom {n}{n/2}$. Similarly the extremal family ${\mathcal B}_0$ for the tilted Sperner conditions described in \cite{LL} have size exactly equal to the weight of the tilted Sperner graph. It is natural to ask 
whether $w(G)$ always determines the size of all maximal $(G,{\mathbf{x}})$-Sperner families? 

In general this is not true (an example is given at the end of Section $2$). However our main result here shows that, with some small control on the values of $x_{ij}$, all 
$(G,\mathbf{x})$-Sperner families ${\mathcal A}$ satisfy $|{\mathcal A}| \leq (1+o(1))w(G)$.

\begin{thm}
\label{genstickout}
 Let $G$ be a graph on vertex set $\{0,\ldots ,n\}$. Suppose that for all edges $ij$ of $G$ with $i<j$,  
$x_{ij}\in \{0,\ldots ,n/2 - 9{(n\log n)}^{1/2}\}$. Then all $(G,\mathbf{x})$-Sperner families ${\mathcal A}$ satisfy 
 \begin{equation}
   \label{boundforGxSperner}
   |{\mathcal A}| \leq w(G) + \frac{C}{n^{2/3}}2^n = (1 + o(1)) w(G).
 \end{equation}
\end{thm}

\noindent \emph{Remark:} While the condition on the values of $x_{ij}$ in Theorem \ref{genstickout} may seem artificial, an example will be given in the next section to show that in general, in order for the conclusion of the theorem to hold, it is necessary that $x_{ij} \in [0, n/2 -c(n\log n)^{1/2}]$ for some $c>0$. 

We draw attention to the fact that with $\mathbf{x}$ as in Theorem \ref{genstickout}, the maximum size of a $(G,\mathbf{x})$-Sperner family does not depend on what the pairwise restrictions are between different layers (the values of $x_{ij}$) but just whether there is one (i.e. whether $ij\in E(G)$), which we feel is quite surprising.

Now note that it is easy to see that $\mathbf{x}$ satisfies the conditions of Theorem \ref{genstickout} for both the Sperner and tilted Sperner conditions. Theorem \ref{genstickout} therefore shows that Sperner and tilted Sperner families $\mathcal A$ satisfy $|{\mathcal {A}}|\leq (1+o(1))w(G)$ for the appropriate $G$. Thus Theorem \ref{genstickout} includes Theorem \ref{tiltedsperner} as a special case.
 
Another natural question to ask is what happens if instead of restricting the size of $A\backslash B$ between sets $A$ 
and $B$ in different layers of the cube, we restrict `patterns' between such sets. For example, how large can 
$\mathcal A\subset {\mathcal P}[n]$ be if $\mathcal A$ does not contain two sets $A$ and $B$ with $|B\backslash A| = 2|A\backslash B|$ in 
which $a<b$ for every $a\in A\backslash B$ and $b\in B\backslash A$? This condition is a substantial restriction of the $1:2$-tilted 
Sperner condition from \cite{LL}. Does this still force $|\mathcal A|=o(2^n)$?

Our second result gives a positive answer to this question. It shows that this much weaker condition 
gives almost the same upper bound on $|{\mathcal A}|$ as given by the $1:2$-tilted Sperner condition.

\begin{thm}
\label{ordertilted}
Suppose that $\mathcal A\subset \mathcal {P}[n]$ does not contain sets $A$ and $B$ with $|B\backslash A| = 
2|A\backslash B|$ in which $a<b$ for every $a\in A\backslash B$ and $b\in B\backslash A$. Then $|{\mathcal A}| \leq 
C e^{120(\log n)^{1/2}}2^n/n^{1/2}$ where $C>0$ is an absolute constant.
\end{thm}

Note that the bound in Theorem \ref{ordertilted} shows that $|{\mathcal A}|\leq 2^n/n^{1/2 + o(1)}$, which up to 
the $o(1)$ term is the size of the largest tilted Sperner family. It would be interesting to know whether the 
$e^{120(\log n)^{1/2}}$ factor above can also be removed.

We will give two proofs of this result. The first gives a short proof using the density Hales-Jewett theorem but 
consequently gives an extremely weak upper bound on $|\mathcal A|$. The second proof is more involved but improves 
this to give the bound stated in Theorem \ref{ordertilted}. 

The proof of Theorem \ref{genstickout} is given in Section 2, followed by the proofs of Theorem \ref{ordertilted} in Section 3. We conclude with some open problems. Throughout we omit floor and ceiling signs whenever they 
are not crucial for the sake of clarity. Our notation is standard. We write $[n]$ for the set $\{1,\ldots, n\}$ and $[a,b]$ for the interval $\{a,a+1,\ldots ,b\}$. For a set $S$, ${\mathcal P}(S)$ denotes the power set of $S$ and $S^{(k)} = \{A\subset S: |A|=k\}$ denotes the $k$-sets of $S$.

%
%
%
%
%
% Section 2
%
%
%
%
%
%

\section{$(G, \mathbf{x})$-Sperner families}

Let $T$ be a set of size $t$. Two sets $A_1, A_2 \in T^{(t/2)}$ are said to be \emph{neighbours} if $|A_1 \triangle A_2| = 2$. Note that any set $A \in T^{(t/2)}$ has $t^2/4$ neighbours in $T^{(t/2)}$. Our first lemma shows that if $\mathcal B \subset T^{(t/2)}$ is large then ${\mathcal B}$ contains a large subset ${\mathcal E}$ such that all elements of ${\mathcal E}$ contain many neighbours in ${\mathcal E}$.

\begin{lem}
	\label{neighbourbound}
	Let $T$ be a set of size $t$. Suppose that ${\mathcal B} \subset 
	T^{(t/2)}$. 
	Then, given any $\alpha \in [4/t,1]$, there exists a set 
	${\mathcal E} \subset {\mathcal B}$ with 
	$|{\mathcal E}| > |{\mathcal B}| - \alpha \binom {t}{t/2}$ so that 
	all $E \in {\mathcal E}$ have at least $\alpha t^2/2^5$
	neighbours in ${\mathcal E}$.
\end{lem}

\begin{proof}
Let ${\mathcal E}$ be a maximal subset of ${\mathcal B}$ with the property that every $E\in {\mathcal E}$ has at least $\alpha t^2/2^5$ neighbours in ${\mathcal E}$ and set ${\mathcal D} =  {\mathcal B} \setminus {\mathcal E}$. Let $|{\mathcal D}| = \gamma \binom {t}{t/2}$ and for contradiction suppose that $\gamma \geq \alpha $. 

Given a set $A\in T^{(t/2+1)}$, let $y_A$ denote the number of sets $D\in {\mathcal D}$ contained in $A$. Double counting we have 
\begin{align*}
\sum _{A\in T^{(t/2+1) } }   y_A 
& = \sum _{A \in T^{(t/2 +1)}} \sum _{D \in {\mathcal D}: D \subset A} 1 
= \sum _{D \in {\mathcal D}} \sum _{A \in {T^{(t/2 +1)}: D \subset A}} 1\\
&= \frac{t}{2} | {\mathcal D} | 
=   \frac{\gamma t}{2}  \binom {t}{t/2} > \frac{\gamma  t}{2} \binom {t}{t/2+1}.
\end{align*}
Now note that for each pair $\{B,B'\}$ of neighbours in ${\mathcal D}$ there exists a unique element $A \in T^{(t/2+1)}$ containing $B$ and $B'$. This shows that
\begin{equation*}
| \big \{ \{B,B'\}: B,B'\mbox{ are neighbours in }{{\mathcal D}} \big \}| 
= \sum _{A\in T^{(t/2+1)} }\binom {y_A}{2}.
\end{equation*}
By the convexity of $\binom {x}{2}$ we therefore have
\begin{align}
\begin{split}
\label{countpairsnew} 
| \big \{ \{B,B'\}: B,B'\mbox{ are neighbours in }{{\mathcal D}} \big \}| 
&\geq \binom {\gamma  t/2}{2} \binom {t}{t/2+1} \\
&\geq \frac {(\gamma  t/2)(\gamma t/2 - 1)}{2^2} \binom {t}{t/2}\\
&\geq \frac {\gamma t^2}{2^5} |{\mathcal D}|.
\end{split}
\end{align}
The final inequality here holds since $\gamma t/2 - 1 \geq \gamma t/4$ for $\gamma \geq 4/t$.

Now view ${\mathcal D}$ as the vertices of a graph in which elements are joined if they are neighbours. By \eqref{countpairsnew} we see that the average degree of this graphs is at least $\gamma  t^2/2^4$. But any graph with average degree $d$ contains a subgraph with minimum degree at least $d/2$ (obtained by repeatedly removing vertices of degree less than $d/2$). This gives a non-empty subset $\mathcal S$
of ${\mathcal D}$ in which every element has at least $\gamma t^2/2^5 \geq \alpha t^2/2^5$ neighbours. However ${\mathcal E} \cup {\mathcal S}$ is a subset of ${\mathcal B}$ in which all elements have at least $\alpha t^2/2^5$ neighbours, contradicting the maximality of ${\mathcal E}$. Therefore $\gamma < \alpha$, as claimed. 
\end{proof}

We now give an overview of the proof of Theorem \ref{genstickout}. Let $G$ and 
${\mathbf{x}}$ be as in the statement of Theorem \ref{genstickout} and let $\A \subset {\mathcal P}[n]$ with $|{\mathcal A}| > w(G) + C2^n/n^{2/3}$. We wish to show that there is some edge $ij$ of $G$ with $i<j$, and sets $A_i \in \A ^{(i)}, A_j \in \A ^{(j)}$ such that $|A_i \setminus A_j| = x_{ij}$. To do this we will proceed in two steps. In the first step we find a maximal chain ${\mathcal C} = \{C_i: i\in [0,n]\} \subset {\mathcal P}[n]$, with $|C_i| = i$ for all $i$, with two properties. The first is that there is a `large' subset $\{C_i: i\in I\}$ with $I \subset [0,n]$ of elements in $\A \cap {\mathcal C}$. Here `large' will not mean with respect to the size of $I$, but with respect to a certain weighted measure of $I$. This property will be used to find an edge $ij$ of $G$ with $i<j$ such that $i,j\in I$. The second property we will need from ${\mathcal C}$ is that each element $C_i$ with $i\in I$ satisfies certain local density conditions in ${\mathcal A}^{(i)}$. The second step of the argument then uses these local density conditions to find a set $A_i \in \A ^{(i)}$ which is close to $C_i$ with $|A_i\setminus C_j| = x_{ij}$.

The first step of the argument will be carried out in the following lemma which locates the chain ${\mathcal C}$ mentioned above. In the statement of the lemma \emph{(i)} and \emph{(ii)} correspond to the property that $I$ is `large' mentioned above. The slightly technical \emph{(iii)}, \emph{(iv)} and \emph{(v)} then correspond to the local density property mentioned above. Each of these will be used to deal with a different range of $x_{ij}$. The reader may find it helpful to skip the proof of the lemma on first reading to see how the conditions \emph{(i)}-\emph{(v)} are used in the proof of Theorem \ref{genstickout}.

Below we make the convention that given a set $A\subset [n]$ and a permutation $\sigma \in S_n$, $\sigma (A)$ denotes the set $\{ \sigma (a): a\in A\}$. 

\begin{lem}
 \label{configuration}
 Given any family ${\mathcal A} \subset {\mathcal P}[n]$ there exist
 \begin{itemize}
 	\item sets $S\subset T\subset [n]$ with $|S| = s = 4n^{3/5}$, 
 	$|T| = t = n - 2(n\log n)^{1/2}$; 
 	\item a maximal chain ${\mathcal C} = \{C_i: i\in [0,n]\}\subset 
 	{\mathcal P}[n]$  with $|C_i| = i$ for all $i$; 
 	\item a set $I \subset [0,n]$;
 \end{itemize}
 with the following properties.
\begin{enumerate}[(i)]
 \item $C_i\in {\mathcal A}$ for all $i\in I$.

 \item $\sum _{i\in I} \binom {n}{i} \geq |{\mathcal A}| - \frac{2^{10}}{n^{2/3}}2^n$.

 \item There is $F_0\in S^{(s/2)}$ such that $C_i \cap S = F_0$
       for all $i\in I$ (we will have $i\geq s/2$ for $i\in I$). Furthermore, for each such $i$ there is a family 
       ${\mathcal F}_i \subset S^{(s/2)}$ with $|{\mathcal F}_i| \geq 
       \frac{1}{n} \binom {s}{s/2}$ such that 
       $(C_i\setminus F_0)\cup F \in {\mathcal A}$ for all 
       $F\in {\mathcal F}_i$. 

 \item There is $D_0 \in T^{(t/2)}$ such that $C_i \cap T = D_0$ for all $i\in I$ (we will have $i\geq t/2$ for $i\in I$). Furthermore, for each such $i$ there is a family 
       ${\mathcal D}_i \subset T^{(t/2)}$ with $|{\mathcal D}_i| \geq \frac{1}{n} 
       \binom {t}{t/2}$ such that $(C_i\setminus D_0)\cup D \in {\mathcal A}$ for all 
       $D\in {\mathcal D}_i$. 

 \item For all $i\in I$, every element of ${\mathcal D}_i$ has at least $n^{4/3}$ neighbours 
       in ${\mathcal D}_i$.
\end{enumerate}

\end{lem}

\begin{proof}
We may assume that $n>2^{15}$ as otherwise, taking $I=\emptyset$ the statement of the lemma is vacuous. Let us set $n_1 = 2 (n\log n)^{1/2}$ so that $n = t + n_1$ and let $I_0 = [n/2-(n\log n)^{1/2}, n/2+(n\log n)^{1/2}] = [t/2, t/2 + n_1] \subset [n]$. We will restrict to those elements $\widetilde {\mathcal A} \subset {\mathcal A}$ with $\widetilde {\mathcal A} = \{A \in {\mathcal A}: |A| \in I_0\}$. We have $|{\widetilde {\mathcal A}}| \geq |{\mathcal A}| - 2^{n}/n$ by  Chernoff's inequality (see Appendix A, \cite{aands}). We let $|\widetilde {\mathcal A}^{(i)}| = \alpha _i \binom {n}{i}$.

To begin, choose a permutation $\sigma \in S_n$ uniformly at random. Note that for $B \in [1,t]^{(t/2)}$ and $i\in I_0$ we have $|\sigma (B \cup [t+1,i+t/2])| = i$  (here we take $[t+1,t]= \emptyset$ when $i = t/2$). Let ${\mathcal B}_i \subset {[1,t]}^{(t/2)}$ denote those sets $B \in {[1,t]}^{(t/2)}$ with $\sigma (B \cup [t+1,i + t/2]) \in \widetilde {\mathcal A}^{(i)}$ and write 
$|{\mathcal B}_i| = \beta _i \binom {t}{t/2}$. Also let $X_i$ denote the random variable given by 
\[X_i = \left\{ \begin{array}{ll}
         |{\mathcal B}_i| & \mbox{if $\beta _i > \frac{2^7}{n^{2/3}}$};\\
        0 & \mbox{otherwise.}\end{array} \right. \]

We claim that $\mathbb {E}(X_i) \geq (\alpha _i - 2^7/n^{2/3})\binom {t}{t/2}$. Indeed, as $\sigma $ is chosen uniformly at random, $\sigma (B \cup [t+1,t/2 +i])$ is equally likely to be any set in $[n]^{(i)}$, giving
 \begin{equation*}
  \mathbb{P}(\sigma (B\cup [t + 1, t/2 + i]) \in {\widetilde \A }^{(i)}) = \frac {|{\widetilde \A }^{(i)}|}
  {\binom {n}{i}} = \alpha _i.
 \end{equation*}
This gives that $ \mathbb{E}(|{\mathcal B}_i|) = \sum _{B \in [1,t]^{(t/2)}} \mathbb {P}(\sigma (B\cup [t + 1, t/2+i]) \in {\widetilde \A} ^{(i)}) = \alpha _{i} \binom {t}{t/2}$. Using that $X_i \geq |{\mathcal B}_i| - \frac{2^7}{n^{2/3}}\binom {t}{t/2}$ for all $i$ then gives 
 \begin{equation}
  \label{expectationofXi} 
  \mathbb{E}(X_i) \geq  \alpha _{i} \binom{t}{t/2} - \frac{2^7}{n^{2/3}} \binom{t}{t/2},
 \end{equation}
 proving the claim.

 In order to guarantee \emph{(ii)} we will make a choice of $\sigma $ according to a certain weighted function. Let $Z$ denote the random variable $Z = \sum _{i\in I_0} \frac {\binom {n}{i}} {\binom {t}{t/2}}X_i$. Using linearity of expectation and \eqref{expectationofXi} we have
\begin{align}
\begin{split}
\label{Zexpectation}
 \mathbb{E}(Z) & = \sum _{i\in I_0} \frac{\binom {n}{i}}{\binom {t}{t/2}} {\mathbb E}(X_i) \geq \sum _{i\in I_0} (\alpha _{i} - \frac{2^7}{n^{2/3}})\binom {n}{i} \\
  & = \sum _{i\in I_0} |\widetilde {\mathcal A}_i| - \frac{2^7}{n^{2/3}} \binom {n}{i} \geq 
 |\widetilde {\mathcal A}| - \frac{2^{n+7}}{n^{2/3}} \geq |{\mathcal A}| - \frac{2^{n+8}}{n^{2/3}}.
 \end{split}
\end{align}
Now fix a choice of $\sigma \in S_n$ so that $Z(\sigma ) \geq {\mathbb E}(Z)$. Take $I_1 \subset I_0$ to consist of those $i\in I_0$ with $X_i \neq 0$. By \eqref{Zexpectation} this gives 
\begin{equation}
 \label{weightedsummationB}
	\sum _{i\in I_1} \beta _i \binom {n}{i} 
	= \sum _{i\in I_0} \frac {\binom {n}{i}}{\binom {t}{t/2}}X_i(\sigma ) = Z(\sigma ) \geq {\mathbb E}(Z) \geq |{\mathcal A}| - \frac{2^{n+8}}{n^{2/3}}.
\end{equation} 
Furthermore, by definition of $X_i$ we have $\beta _i > \frac{2^7}{n^{2/3}}$ for $i\in I_1$.  

We now use Lemma \ref{neighbourbound}, taking $\alpha = 2^6/n^{2/3}$, to find a set ${\mathcal E}_i \subset {\mathcal B}_i$ such that all elements of ${\mathcal E}_i$ have many neighbours in ${\mathcal E}_i$. This gives a family ${\mathcal E}_i \subset {\mathcal B}_i \subset [1,t]^{(t/2)}$ with $|{\mathcal E}_{i}| = \delta _i\binom {t}{t/2}$ satisfying
\begin{equation}
	\label{delta_i density bound}
|{\mathcal E}_i| = \delta _i \binom {t}{t/2} \geq |{\mathcal B}_i| - \frac {2^6}{n^{2/3}} \binom {t}{t/2} = (\beta _i - \frac{2^6}{n^{2/3}})\binom {t}{t/2} \geq \frac{2^6}{n^{2/3}}\binom {t}{t/2},
\end{equation} 
as $\beta _i > 2^{7}/{n^{2/3}}$ for all $i\in I_1$ and so that each $E \in {\mathcal E}_i$ has at least $\alpha t^2 /2^5 \geq 2(n - 2(n\log n)^{1/2})^2/n^{2/3} \geq n^{4/3}$ neighbours in ${\mathcal E} _i$.

Now let $n_2$ be such that $t = s + 2n_2$. Choose a permutation 
$\pi\in S_t$ of the elements of $[1,t]$ uniformly at random.
For each $i\in I_1$, we write ${\mathcal G}_i \subset [1,s]^{(s/2)}$ for the 
collection of sets $G \in [1,s]^{(s/2)}$ such that $\pi (G \cup [s+1,s+n_2]) \in {\mathcal E}_i$ -- note again that $|\pi (G \cup [s+1,s+n_2])| = t/2$ for all $G \in [1,s]^{(s/2)}$. As $\pi $ is chosen uniformly at random we have 
\begin{equation*}
	{\mathbb E}(|{\mathcal G}_i|) = \frac {|{\mathcal E}_i|}{\binom {t}{t/2}}\binom {s}{s/2} = \delta _i \binom {s}{s/2}.
\end{equation*}
For each set $G \in [1,s]^{(s/2)}$ let $Y_{i,G}$ denote the indicator random variable given by
\[Y_{i,G} = \left\{ \begin{array}{ll}
         1 & \mbox{if $G \in {\mathcal G}_i$ and $|{\mathcal G}_i| \geq \frac {1}{n} \binom {s}{s/2}$};\\
        0 & \mbox{otherwise.}\end{array} \right. \]

We claim that ${\mathbb E}(Y_{i,[1,s/2]}) \geq \delta _i - \frac {1}{n}$. Indeed, as $\pi \in S_t$ is chosen uniformly at random we have $\mathbb {E}(Y_{i,[1,s/2]}) = \mathbb {E}(Y_{i,G})$ for all $G \in [1,s]^{(s/2)}$. Therefore
\begin{equation*}
	\label{Yi expectation}
	\binom {s}{s/2} {\mathbb E}(Y_{i,[1,s/2]})  = 
	\sum _{G \in [1,s]^{(s/2)}} {\mathbb E}(Y_{i,G})
	 \geq 
	{\mathbb E}(|{\mathcal G}_i|) - \frac{1}{n} \binom {s}{s/2}
	 \geq 
	(\delta _i - \frac {1}{n}) \binom {s}{s/2},
\end{equation*}
which after dividing by $\binom {s}{s/2}$ gives the claim. 

Now consider the random variable $W = \sum _{i\in I_1} \binom {n}{i} Y_{i,[1,s/2]}$. By the previous claim, we have 
\begin{align}
 \begin{split}
 \label{Wexpectation}
  \mathbb{E}(W) & = \sum _{i\in I_1} \binom {n}{i} {\mathbb E}(Y_{i,[1,s/2]}) \geq \sum _{i\in I_1} (\delta _i - \frac{1}{n}) \binom {n}{i}\\
  & \geq \sum _{i\in I_1} (\beta _i - \frac{2^6}{n^{2/3}} - \frac{1}{n})\binom {n}{i}
  \geq |{\mathcal A}| - \frac{2^{10}}{n^{2/3}}2^n.
 \end{split}
\end{align}
The second inequality here follows since $\delta _i \geq \beta _i - 2^{6}/{n^{2/3}}$ by \eqref{delta_i density bound} and the third inequality follows by \eqref{weightedsummationB}. Fix a choice of $\pi \in S_t$ such that $W(\pi ) \geq |{\mathcal A}| - \frac{2^{10}}{n^{2/3}}2^n$.

We now define the sets in the statement. Let $S = \sigma ( \pi ([1,s]))$ and $T = \sigma ([1,t])$. Since $\pi \in S_t$ we have $\pi ([1,s]) \subset [1,t]$ and so $S \subset T$. Let $I = \{i\in I_1: Y_{i,[1,s/2]}(\pi) = 1\}$ and set $F_0 = \sigma (\pi ([1,s/2]))$ and $D_0 = F_0 \cup \sigma (\pi ([s+1,s+n_2]))$. For all $i\in I$ let 
\begin{align*}
C_i = F_0 \cup \sigma (\pi ([s+1,s+n_2])) \cup \sigma ([t +1, t/2 + i])  = D_0 \cup \sigma ([t +1, t/2 + i]).
\end{align*} 
Take ${\mathcal C}$ to be any maximal chain extending this partial chain. Lastly, set ${\mathcal F}_i = \sigma (\pi ({\mathcal G}_i)) = \{\sigma (\pi (G)): G \in {\mathcal G}_i\}$ and ${\mathcal D}_i = \sigma ({\mathcal E}_i) = \{\sigma (E): E \in {\mathcal E}_i\}$.

All that now remains is to verify that \emph{(i)}-\emph{(v)} are satisfied for these sets. To see \emph{(i)}, note that by definition of $Y_{i,[1,s/2]}$, $[1,s/2] \in {\mathcal G}_i$ for $i\in I$ and therefore $\pi ([1,s/2]) \cup \pi ([s+1,s +n_2]) \in {\mathcal E}_i \subset {\mathcal B}_i$, giving (by definition of ${\mathcal B}_i$) that $C_i \in {\widetilde {\mathcal A}}^{(i)} \subset {\mathcal A}^{(i)}$. Furthermore, \emph{(ii)} follows from \eqref{Wexpectation} and our choice of $\pi $ since 
\begin{equation*}
	\sum _{i\in I} \binom {n}{i} = \sum _{i\in I_1} \binom {n}{i} Y_{i,[1,s/2]}(\pi ) = W (\pi ) \geq {\mathbb E}(W) \geq |\A | - \frac{2^{10}}{n^{2/3}}2^n.
\end{equation*}
To see \emph{(iii)}, first note that $C_i \cap S = C_i \cap \sigma (\pi ([1,s])) = \sigma (\pi ([1,s/2])) = F_0$. Also, for $F \in {\mathcal F}_i$ we have $F = \sigma (\pi (G))$ for some $G \in {\mathcal G}_i$ and 
\begin{equation*}
(C_i \setminus F_0) \cup F = \sigma ( \pi (G \cup [s + 1, s + 
n_2])) \cup \sigma ([t + 1, t/2 + i]).
\end{equation*} 
By definition of ${\mathcal G}_i$ we have $\pi (G \cup [s + 1, s + 
n_2]) \in {\mathcal E}_i \subset {\mathcal B}_i$ and by definition of ${\mathcal B}_i$ this gives $(C_i \setminus F_0) \cup F \in {\mathcal A}^{(i)}$. As $Y_{i,[1,s/2]} = 1$ for all $i\in I$, we also have $|{\mathcal F}_i| = |\sigma (\pi ({\mathcal G}_i))| = |{\mathcal G}_i| \geq \frac {1}{n} \binom {s}{s/2}$, which gives \emph{(iii)}. To see \emph{(iv)} note that that $C_i \cap T = \sigma (\pi ([1,s/2] \cup [s+1,s+n_2])) = D_0$ and if $D \in {\mathcal D}_i$ with $D = \sigma (E)$ for some $E\in {\mathcal E}_i$ we have $(C_i \setminus D_0) \cup D = \sigma (E) \cup \sigma ([t+1,t/2 +i]) \in \A ^{(i)}$. We also have $|{\mathcal D}_i| = |\sigma ({\mathcal E}_i)| = \delta _i\binom {t}{t/2} \geq 2^6/n^{2/3} \binom{t}{t/2} > 1/n\binom {t}{t/2}$ completing \emph{(iv)}. Lastly, by construction \emph{(v)} holds for the family ${\mathcal E}_i$ and therefore also holds for $\sigma ({\mathcal E}_i) = {\mathcal D}_i$. This completes the proof of the lemma.
\end{proof}

The proof of Theorem \ref{genstickout} will also make use of the following 
powerful theorem of Frankl and R\"odl from \cite{forbid} (see Theorem 1.4 in \cite{forbid}, or \cite{kl} for an alternative proof).

\begin{thm}[Frankl and R\"odl]
\label{franklrodl}
{
Let $0<\eta <1/4$ and let $l$ be an integer with $\eta n \leq l \leq (1/2 - \eta )n$. Suppose that ${\mathcal A}, {\mathcal B}\subset {\mathcal P}[n]$ with $|A\cap B| \neq l$ for all $A\in {\mathcal A}$, $B\in {\mathcal B}$. Then $|{\mathcal A}||{\mathcal B}|\leq (4-\epsilon )^n$, where $\epsilon = \epsilon (\eta ) > 0$.
}
\end{thm}

\noindent \emph{Remark:} In particular, the proof of Theorem \ref{franklrodl} in \cite {forbid} 
gives that for $\eta = 1/10$ we can take $\epsilon = 1/400$ above. Frankl and R\"odl also showed that if $l = \rho n$ with $\rho \in [0,1]$, then $|{\mathcal A}||{\mathcal B}| \leq (4-\rho ^2 + O(\rho ^3))^n$ 
(see Corollary 2.4 in \cite{forbid}). In particular, for $\rho \in [0,1/10]$, $|{\mathcal A}||{\mathcal B}|\leq e^{-\rho ^{2}n/16}4^n = e^{-l ^{2}/16n}4^n$. Combining these two ranges shows that if $l \in [0 ,n/3]$ and ${\mathcal A}, {\mathcal B}\subset {\mathcal P}[n]$ with $|A\cap B| \neq l$ for all $A\in {\mathcal A}$, $B\in {\mathcal B}$ then 
\begin{equation}
   \label{pocketsized Frankl Rodl bound}
	|{\mathcal A}||{\mathcal B}| \leq \max \{(4 - 1/400)^{n}, e^{-l^2/16n}4^n\} = 
	\max \{(1 - 1/1600)^{n}, e^{-l^2/16n} \}4^n.
\end{equation}

We are now ready to complete the proof of Theorem \ref{genstickout}.
%
%
%
% Proof of Main Theorem
%
%
%

\begin{proof}[Proof of Theorem \ref{genstickout}.] 
 We will prove the theorem with $C = 2^{200}$. We may assume that $n\geq 2^{300}$ since 
 otherwise $\frac{C}{n^{2/3}} 2^n > 2^n$ and the conclusion is trivial.
 Let $G$ and $\mathbf{x}$ be as in the statement of the theorem and suppose for 
 contradiction that ${\mathcal A}$ is a \mbox {$(G,\mathbf{x})$-Sperner} family with 
 $|{\mathcal A}|> w(G) + \frac{C}{n^{2/3}}2^n$. 
 
 To begin, apply Lemma \ref{configuration} to ${\mathcal A}$ 
 to find sets $S$ and $T$, a chain ${\mathcal C}$ and a set $I \subset [0,n]$ as 
 in the Lemma. Now by Lemma \ref{configuration} \emph{(ii)} 
 $\sum _{i\in I} \binom {n}{i} \geq |{\mathcal A}| -  \frac{2^{10}}{n^{2/3}}2^n > w(G)$. By definition of $w(G)$, $I$ cannot be an independent set of $G$. 
 Therefore there exist $ij\in E(G)$ with $i, j\in I$. Now note that Lemma \ref{configuration} \emph {(i)} guarantees that
 $C_i$ and $C_j$ are in ${\mathcal A}$. We will show that regardless of the value of $x_{ij}$ we can 
 find sets in ${\mathcal A}^{(i)}$ and ${\mathcal A}^{(j)}$ which violate the $(G,{\mathbf{x}})$-Sperner 
 condition.

\noindent \textbf{Case I:} $x_{ij} \in [0, n^{1/3}]$.

 Starting with $D_0$ as in Lemma \ref{configuration} \emph {(iv)}, we will construct a sequence of sets 
 $D_0, D_1,\ldots ,D_{x_{ij}} \in {\mathcal D}_i$ such that each consecutive pair $D_l$ and $D_{l+1}$ 
 are neighbours and $|D_{l+1}\backslash {D_0}| = |D_l\backslash {D_0}| + 1$ for all $l\in[0,x_{ij}-1]$. 
 This will then give that $(C_i \setminus D_0) \cup D_{x_{ij}}\in {\mathcal A}^{(i)}$ and 
 $C_j \in {\mathcal A}^{(j)}$. But since $C_i \subset C_j$ and $C_i \cap T = C_j \cap T = D_0$ 
 \begin{equation*}
  |((C_i \setminus D_0) \cup D_{x_{ij}})\setminus C_j| = |D_{x_{ij}} \setminus D_0 | = x_{ij}
 \end{equation*}
 which contradicts the $(G,\mathbf{x})$-Sperner condition.

 Suppose we have so far found sets $D_0,\ldots ,D_k$ with $k< x_{ij}$ and now wish to pick $D_{k+1}$. 
 A neighbour $E$ of $D_k$ belonging to ${\mathcal D}_i$ can be taken as $D_{k+1}$ so long as 
 $E\setminus D_k \not \subset D_0$ and $D_k\setminus E \subset D_0$. 
 But there are at most $|D_0 \setminus D_k |t/2$ neighbours of $D_k$ in ${\mathcal D}_i$ which fail 
 to satisfy the first condition and at most $|D_k\cap (T \setminus D_0)|t/2$ 
 which fail to satisfy the second. Now 
\begin{equation*}
 \label{badneighbourcount}
  \begin{split}
\frac{|D_0 \setminus D_k |t}{2} + \frac{|D_k\cap (T \setminus D_0)|t}{2} 
 = \frac{kt}{2} + \frac{kt}{2}%\\
 < x_{ij}n %\\
 \leq n^{4/3}.
\end{split}
\end{equation*}
Now as $D_k \in {\mathcal D}_i$, by Lemma \ref{configuration} 
\emph{(v)} $D_k$ has at least $n^{4/3}$ neighbours in ${\mathcal D}_i$ and therefore there is a suitable choice for $D_{k+1}$, as required.

\noindent \textbf{Case II:} $x_{ij} \in [n^{1/3}, n^{3/5}]$.

Since $i,j \in I$, by Lemma \ref{configuration} \emph{(iii)} we have $C_i \cap S = C_j \cap S = F_0$. By the $(G,\mathbf{x})$-Sperner condition, we must have $|A \setminus B| \neq x_{ij}$ for all $A \in {\mathcal A}^{(i)}$ and $B \in {\mathcal A}^{(j)}$. However, also by Lemma \ref{configuration} \emph{(iii)}, $(C_i \setminus F_0) \cup F \in {\mathcal A}^{(i)}$ and $(C_j\setminus F_0) \cup F' \in {\mathcal A}^{(j)}$ for all $F \in {\mathcal F}_i$ and $F' \in 
{\mathcal F}_j$. This gives that 
\begin{equation}
	\label{intersection bound for F_i and F^*_j}
	|F \cap (S\setminus F')| = |F \setminus F'| = |((C_i \setminus F_0) \cup F )\setminus ((C_j \setminus F_0) \cup F')| \neq x_{ij}
\end{equation}
and also by Lemma \ref{configuration} \emph{(iii)}, ${\mathcal F}_i$ and ${\mathcal F}_j$ satisfy
\begin{equation}
 \label{boundsonsizeofFiandFj}
  |{\mathcal F}_i|, |{\mathcal F}_j| \geq \frac{1}{n}\binom {s}{s/2} 
  \geq \frac{1}{n^2}2^s.
\end{equation}
We now show that the Frankl-R\"odl theorem contradicts \eqref{boundsonsizeofFiandFj}.

Let ${\mathcal F}^* _j $ denote the set $\{S\setminus F: F\in {\mathcal F}_j\} \subset S^{(s/2)}$. Now ${\mathcal F}_i, {\mathcal F}^*_j \subset {\mathcal P}(S)$ and by \eqref{intersection bound for F_i and F^*_j} we have $|F \cap F'| \neq x_{ij}$ for all $F \in {\mathcal F}_i$ and $F' \in {\mathcal F}^*_j$. As  $x_{ij} \in [0,n^{3/5}] = [0,|S|/4]$, by \eqref{pocketsized Frankl Rodl bound}
\begin{align*}
 \begin{split}
 |{\mathcal F}_i||{\mathcal F}_j| = |{\mathcal F}_i||{\mathcal F}^*_j| 
 & \leq \max \{ (1-{1/1600})^{s}, e^{-{x_{ij}}^2/{16s}} \} 4^{s}\\
 & \leq \max \{ (1-{1/1600})^{s}, e^{-{n^{2/3}}/{64n^{3/5}}}\} 4^{s}\\
 & = \max \{ (1-{1/1600})^{4n^{3/5}}, e^{-{n^{1/15}}/{64}}\} 4^{s}\\
 & < \frac{1}{n^{4}}4^s.
 \end{split}
\end{align*} 
The second inequality here holds since $x_{ij} \geq n^{1/3}$ and the final inequality holds for $n\geq 2^{300}$. However, this contradicts \eqref{boundsonsizeofFiandFj}.

\noindent \textbf{Case IIIa:} $x_{ij} \in [n^{3/5}, n/4]$.

This is similar to the previous case. Since $i,j \in I$, by Lemma \ref{configuration} \emph{(iv)} we have $C_i \cap T = C_j \cap T = D_0$. By the $(G,\mathbf{x})$-Sperner condition, we must have $|A \setminus B| \neq x_{ij}$ for all $A \in {\mathcal A}^{(i)}$ and $B \in {\mathcal A}^{(j)}$. However, also by Lemma \ref{configuration} \emph{(iv)}, $(C_i \setminus D_0) \cup D \in {\mathcal A}^{(i)}$ and $(C_j\setminus D_0) \cup D' \in {\mathcal A}^{(j)}$ for all $D \in {\mathcal D}_i$ and $D' \in 
{\mathcal D}_j$. This gives that 
\begin{equation}
	\label{intersection bound for D_i and D^*_j}
	|D \cap (T\setminus D')| = |D \setminus D'| = |((C_i \setminus D_0) \cup D )\setminus ((C_j \setminus D_0) \cup D')| \neq x_{ij}
\end{equation}
and also by Lemma \ref{configuration} (iv), ${\mathcal D}_i$ and ${\mathcal D}_j$ satisfy
\begin{equation}
 \label{boundsonsizeofDiandDj}
  |{\mathcal D}_i|, |{\mathcal D}_j| \geq \frac{1}{n}\binom {t}{t/2} 
  > \frac{1}{n^2}2^t.
\end{equation}

Let ${\mathcal D}^* _j $ denote the set $\{T\setminus D: D\in {\mathcal D}_j\} \subset T^{(t/2)}$. Now ${\mathcal D}_i, {\mathcal D}^*_j \subset {\mathcal P}(T)$ and by \eqref{intersection bound for D_i and D^*_j} we have $|D \cap D'| \neq x_{ij}$ for all $D \in {\mathcal D}_i$ and $D' \in {\mathcal D}^*_j$. As  $x_{ij} \in [0,n/4] \subset [0,|T|/3]$, by \eqref{pocketsized Frankl Rodl bound}
\begin{align*}
 \begin{split}
 |{\mathcal D}_i||{\mathcal D}_j| = |{\mathcal D}_i||{\mathcal D}^*_j| 
 & \leq \max \{ (1-{1/1600})^{t}, e^{-{x_{ij}}^2/{16t}} \} 4^{t}\\
 & \leq \max \{ (1-{1/1600})^{t}, e^{-{n^{6/5}}/{64n}}\} 4^{t}\\
 & = \max \{ (1-{1/1600})^{t}, e^{-{n^{1/5}}/{64}}\} 4^{t}\\
 & < \frac{1}{n^{4}}4^t.
 \end{split}
\end{align*} 
The second inequality here holds since $x_{ij} \geq n^{3/5}$ and the final inequality holds for $n\geq 2^{300}$. However, this contradicts \eqref{boundsonsizeofDiandDj}.

\noindent \textbf{Case IIIb:} $x_{ij} \in [n/4, n/2-9(n\log n)^{1/2}]$.

This case can be argued in the same way as Case IIIa by noting that in \eqref{intersection bound for D_i and D^*_j}, we have $|D\setminus D'| \neq x_{ij}$ for all $D\in {\mathcal D}_i$ and $D'\in {\mathcal D}_j$ if and only if $|D \cap D'| \neq t/2 - x_{ij}$. We also have $t/2 - x_{ij} \in [8(n\log n)^{1/2}, n/4] \subset [8(n\log n)^{1/2}, |T|/3]$. By \eqref{pocketsized Frankl Rodl bound} we therefore have
\begin{align*}
\begin{split}
 |{\mathcal D}_i||{\mathcal D}_j| 
 & \leq \max \{ (1-{1/1600})^{t}, e^{-{x_{ij}}^2/{16t}} \} 4^{t}\\
 & \leq \max \{ (1-{1/1600})^{t}, e^{-{64 n\log n}/{16n}} \}4^{t}\\
 & = \max \{ (1-1/1600)^{n-2(n\log n)^{1/2}}, n^{-4} \} 4^{t}\\
 & \leq \frac{1}{n^4} 4^t.
 \end{split}
\end{align*} 
But this again contradicts \eqref{boundsonsizeofDiandDj}.

As the Cases \textbf{I - IIIb} above cover the range of possibilities for the values of $x_{ij}$, this completes the proof to the Theorem.
\end{proof}

\noindent \emph{Remark:} While we have not pursued this here, we note that with a more involved version of 
Lemma \ref{configuration} we can replace the term $\frac{C}{n^{2/3}}2^n$ term appearing in Theorem 
\ref{genstickout} with a term of the form $\frac{C\log n}{n}2^n$. 
\vspace{2mm}

%
%
% End of proof of General Stickout
%
%

We now show that some restriction on the values of $x_{ij}$ as in the statement of 
Theorem \ref{genstickout} is necessary. Indeed, take $G = K_{n+1}$ and let 
$x_{ij}\in [n/2 - \beta n^{1/2}, \min (i,n-j)]$ for all $i<j$. This 
gives $w(G)=(1+o(1))\binom{n}{n/2} = O(\frac{2^n}{n^{1/2}})$. 

Now take ${\mathcal A}\subset {\mathcal P}[n]$ to be the family
\begin{equation*}
 {\mathcal A} = \{ A\subset [n]: |A|\leq n/2 \mbox{ and } |A\cap [n/2]|> n/4 +\beta n^{1/2}/2\}.
\end{equation*}
Clearly we have $|A\cap B|> \beta n^{1/2}$ for all $A\in {\mathcal A}^{(i)}$, $B\in {\mathcal A}^{(j)}$. 
Therefore $|A\setminus B|\leq i - \beta n^{1/2} \leq n/2 - \beta n^{1/2}$. This shows that $\mathcal A$ is a $(G,\mathbf{x})$-Sperner 
family. But it can be shown that for $\beta>1$, $|{\mathcal A}| \geq C^{-\beta ^2} 2^n$ for some fixed $C>1$. 
Now taking $\beta < c (\log n)^{1/2}$ for a small enough $c>0$ gives a 
$(G,\mathbf{x})$-Sperner family of size significantly bigger than $w(G)$.

%
%
%
% Forbidding patterns between layers. Section 3
%
%
%

\section{Forbidding patterns between layers}

As mentioned in the Introduction, our first proof of Theorem \ref{ordertilted} is based on Furstenberg 
and Katznelson's density Hales-Jewett theorem \cite {FK} (see also \cite{DHJ}). A set $L\subset [k]^{n}$ 
is said to be a \emph {combinatorial line} if there exists a partition of $[n] = X_1\cup \cdots \cup X_k \cup A$ 
with $A\neq \emptyset$ such that 
\begin{equation*}
 L=\{(x_1,\ldots ,x_n): x_i = l \mbox{ if } i\in X_l \mbox{ and } x_j=x_k 
\mbox { for all } j,k\in A\}.
\end{equation*}
 The set $A$ is called the \emph {active coordinate set}.

\begin{thm}[Density Hales-Jewett]
 {
  For any $\alpha >0$ and $k\in \mathbb{N}$ there exists $n_0(\alpha ,k)\in \mathbb{N}$ such that if 
  $n\geq n_0(\alpha ,k)$ every set $A\subset [k]^n$ with $|A|\geq \alpha k^n$ contains a combinatorial 
  line.
}
\end{thm}

\begin{proof}[First proof of Theorem \ref{ordertilted}.] It is enough to prove the theorem when $n$ is a multiple of $3$ since the general case 
follows easily from it. Let $n=3m$. We will identify ${\mathcal P}[n] = \{0,1\}^n$ with 
the set $\{0,\ldots ,7\}^m$ via the map $f: \{0,1\}^n \rightarrow \{0,\ldots ,7\}^m$, 
which sends $x=(x_1,\ldots ,x_n)\in \{0,1\}^n$ to $f(x)=(y_1,\ldots ,y_m)$ where 
$y_i=x_i+2x_{i+m}+4x_{i+2m}$ for all $i\in [m]$. 

Suppose $|{\mathcal A}|=\alpha 2^n$ for some constant $\alpha >0$. Then $|f({\mathcal A})|=\alpha 8^m$ 
where $f({\mathcal A})=\{f(a):a\in {\mathcal A}\}$. By the density Hales-Jewett theorem, if $n$ is 
sufficiently large, $f(\mathcal A)$ contains a combinatorial line $L$ with 
$[m]=X_0\cup \ldots \cup X_7\cup A$, where $A$ is the active coordinate set. Write 
$L=\{L_0,\ldots ,L_7\}$ where $L_i$ corresponds to the element of $L$ in which elements 
of the active coordinate set takes the value $i$.

Let $K$ be the subset of $\mathcal A$ which corresponds under $f$ to $L$, i.e $f(K)=L$. We claim 
that $K$ contains a forbidden pair $A$, $B$. Indeed, taking $A\in K$ such that $f(A)=L_1$ 
and $f(B)=L_6$, all elements of $A\backslash B$ occur in $[m]$ while all elements of 
$B\backslash A$ occur in $[m+1,3m]$. Furthermore, for each element $i\in [m]$ in 
$i\in A\backslash B$ if and only if $i+m, i+2m\in B\backslash A$. Therefore 
$|B\backslash A|=2|A\backslash B|$, a contradiction. \end{proof}

%
%
%
%
%
% Second proof of pattern
%
%
%
%
%
%

Our second proof of Theorem \ref{ordertilted} is again given by a Katona type averaging argument (see \cite{katona}). However this time it is more involved, owing to the fact that sets in the same level may forbid a different number of elements in ${\mathcal P}[n]$. For example, if the set $[1,n/3] \in \A ^{(n/3)}$, it forbids many elements of $[n]^{(n/2)}$ from being in $\A ^{(n/2)}$ -- all sets of the form $B \cup C$ where $B \in [1,n/3]^{(n/6)}$ and $C \in [n/3 +1 ,n]^{(n/3)}$. However, if the set $[2n/3 + 1, n] \in {\mathcal A}^{(n/3)}$ it does not prevent any sets from $[n]^{(n/2)}$ lying in ${\mathcal A}^{(n/2)}$. To compensate for this imbalance, we first break the set system ${\mathcal A}$ into smaller pieces all of which behave similarly, in the sense that if two elements of ${\mathcal A}$ lie inside the same piece then they forbid roughly the same number of elements in any other piece. We then carry out a Katona type averaging procedure over these pieces.
\vspace {1.5mm}

%
%
%
% Beginning of second proof of ordered tilted Sperner
%
%
%

%
%
%
%
%
%
%
%  SECOND PROOF
%
%
%
%
%
%
%

\begin{proof}[Second proof of Theorem \ref{ordertilted}.]  We will again assume that $n$ is a multiple of $3$. For convenience, we let $\alpha (n) = e^{120(\log n)^{1/2}}$. We will prove that any set ${\mathcal A} \subset {\mathcal P}[n]$ as in the statement of the theorem satisfies 
\begin{equation*}
	|{\mathcal A}| \leq \frac {100 \alpha (n) }{n^{1/2}}2^n.
\end{equation*}
We can assume that $n\geq 10^4$ as otherwise $100 \alpha (n) \geq n^{1/2}$ and the result is immediate. 

Given a set $D\subset [n]$, let $r_D = |D\cap [n/3]| - n/6 $ and $s_D = |D\cap [n/3+1, n]| - n/3$. Take ${\mathcal B}$ to be the 
subset of ${\mathcal A}$ with
\begin{equation*}
 {\mathcal B} = \{ A\in {\mathcal A}:  |r_A|\leq (n\log n)^{1/2} \mbox { and } |s_A|\leq (n\log n)^{1/2} \}
\end{equation*}
From Chernoff's inequality we have $|{\mathcal A}\setminus {\mathcal B}| \leq (2 n^{-6} + 2 n^{-3})2^n 
\leq \frac{4 \alpha (n) }{n^{1/2}}2^n$
so it suffices to show that 
\begin{equation}
\label{mainsetbound}
 |{\mathcal B}| 
\leq \frac{96 \alpha (n) }{n^{1/2}}2^n.
\end{equation}
Let $L = n^{1/2}$. For all $i,j \in [-(\log n)^{1/2}/2, (\log n)^{1/2}/2]$ we let 
\begin{equation*}
[n]_{i,j} := \{D\subset [n]: |r_D - 2iL|\leq L \mbox{ and } |s_D -2jL| \leq L\}. 
\end{equation*} 
Also let ${\mathcal B}_{i,j} = {\mathcal B}\cap [n]_{i,j}$. To prove \eqref{mainsetbound} it clearly suffices to show 
that
\begin{equation}
 \label{zonebound}
 | {\mathcal B}_{i,j} | 
\leq  \frac {96 \alpha (n)}{n^{1/2}}\big |[n]_{i,j}\big |.
\end{equation}

%
%
%
% Beginning of local bounding
%
%
%

We will fix $i,j \in [-(\log n)^{1/2}/2, (\log n)^{1/2}/2]$ for the remainder of the theorem and show \eqref{zonebound}. Let $K = n^{1/2}/12$ and pick the following:
\begin{itemize}
 \item a set $U\subset [n/3]$ of size $K$ chosen uniformly at random,

 \item a random set $S_1\subset [n/3]\setminus U$ where each $s\in [n/3]\setminus U$ is included in $S_1$ 
       independently with probability $p_{1,i} = 1/2 + 6i/n^{1/2}$, 

 \item a set $V\subset [n/3+1,n]$ of size $2K$ chosen uniformly at random,

 \item a random set $S_2\subset [n/3]\setminus V$ where each $s\in [n/3+1,n]\setminus V$ is included in $S_2$ 
       independently with probability $p_{2,j} = 1/2 + 3j/n^{1/2}$.
\end{itemize}
Finally place a random ordering $(u_1,\ldots ,u_{K})$ on the elements of $U$ and a random ordering $(v_1,\ldots ,v_{2K})$ on the elements of $V$. For all $k\in [0,K]$, let $U_k = \{u_1,\ldots ,u_{k}\}$ and $V_k = \{v_{2K-2k+1},\ldots , v_{2K}\}$.

Having made these choices, for all $k\in [0,K]$ take 
\begin{equation*}
C_k = U_k \cup S_1 \cup V_k \cup S_2
\end{equation*}
and let ${\mathcal C}= \{C_k: k\in[0, K]\}$. Note that any two elements of ${\mathcal C}$ form a forbidden pair. Indeed, for $k< l$ with $k,l \in [0,K]$ we have $C_k\setminus C_l = \{u_{k+1},\ldots , u_l\}$ and $C_l\setminus C_k = \{v_{2K - 2l + 1},\ldots , v_{2K - 2k}\}$ and $u<v$ for all elements $u\in U$, $v\in V$. Therefore, letting $X_k$ be the indicator random variable which equals $1$ if $C_k\in \mathcal B_{i,j}$ and $0$ otherwise, for all choices of ${\mathcal C}$ we have
\begin{equation}
\label{countingidentity}
 \sum _{k=0}^K X_k = |{\mathcal B_{i,j}}\cap {\mathcal C}|\leq 1.
\end{equation}
Therefore taking the expectation of both sides of \eqref{countingidentity} and expanding we have
\begin{equation}
 \sum _{k=0}^{K} \sum _{B\in {\mathcal B_{i,j}}} {\mathbb{P}}(C_k=B)
= \sum _{k=0}^{K}{\mathbb{P}}(C_k\in {\mathcal B_{i,j}})
= \sum _{k=0}^{K}{\mathbb{E}}(X_k)\leq 1.
\label{summationprobability}
\end{equation}
In Lemma \ref{lemma for ordered tilted} which follows we will show that, for all $B\in [n]_{i,j}$ we have
\begin{equation}
 \label{localprobability}
 {\mathbb P}(C_k=B)
 \geq \frac {1}{8\alpha (n)  \big |[n]_{i,j}\big |}.
\end{equation}
We claim that this proves \eqref{zonebound}. Indeed, by \eqref{summationprobability} and \eqref{localprobability} we have
\begin{equation*} 
\frac {|{\mathcal B}_{i,j}|K}{8\alpha (n) \big |[n]_{i,j}\big |}
=
\sum _{k=0}^{K} \sum _{B\in {\mathcal B_{i,j}}} \frac {1}{8\alpha (n) \big |[n]_{i,j}\big |} 
\leq
\sum _{k=0}^{K} \sum _{B\in {\mathcal B_{i,j}}} {\mathbb{P}}(C_k=B)
\leq 1.
\end{equation*}
Using that $K = n^{1/2}/12$ and rearranging, this gives \eqref{zonebound}. \end{proof}

%
%
% Showing that it is enough to prove the local probability bound
%
%
%

\begin{lem}
	\label{lemma for ordered tilted}
	Given $B \in [n]_{i,j}$ and $k\in [0,K]$ we have 
	\begin{equation*}
		{\mathbb P}(C_k=B) \geq \frac {1}{8\alpha (n)  \big |[n]_{i,j}\big |}.
	\end{equation*}
\end{lem}

\begin{proof}
We will show the claimed bound in two steps. In the first we show that ${\mathbb P}(C_k \in [n]_{i,j}) \geq 1/8$. In the second we show that for all $A,B \in [n]_{i,j}$ we have 
\begin{equation} 
	\label{comparingelementsinsamezone}
	\frac {{\mathbb P}(C_k = A)}{{\mathbb P}(C_k = B)} \leq \alpha (n).
\end{equation}
Combined these then gives the bound claimed in the statement of the lemma since 
\begin{align*}
	\frac{1}{8} & \leq {\mathcal P}(C_k \in [n]_{i,j}) = 
	\sum _{A \in [n]_{i,j}} {\mathcal P}(C_k = A) \\
	& \leq \sum _{A \in [n]_{i,j}} \alpha (n) {\mathcal P}(C_k = B)
	= \alpha (n) |[n]_{i,j}| {\mathcal P}(C_k = B).
\end{align*}

We first show the bound on ${\mathbb P}(C_k \in [n]_{i,j})$. Now $C_k = (U_k \cup S_1) \cup (V_k \cup S_2) \in [n]_{i,j}$ if 
\begin{align*}
\begin{split}
 & |U_k \cup S_1| = k + |S_1| \in [n/6 + (2i-1)L, n/6 + (2i + 1)L-1] \mbox{ and }\\
 & |V_k \cup S_2| = 2K - 2k + |S_2| \in [n/3 + (2j-1)L, n/3 + (2j+1)L-1].	  
\end{split}
\end{align*}
Using that $k\in [0,K]$ we therefore find that $C_k\in [n]_{i,j}$ if 
\begin{equation}
 \label{conditionforS_1}
 |S_1|\in [n/6 +(2i-1)L ,  n/6 + (2i + 1)L - K - 1]
\end{equation}
and
\begin{equation}
 \label{conditionforS_2}
|S_2| \in [n/3 + (2j-1)L, n/3 + (2j + 1)L-2K - 1].	
\end{equation}
But $|S_1|$ and $|S_2|$ are binomially distributed, with $|S_1| \sim \operatorname{B}\left({n/3-K,1/2 + 6i/n^{1/2}}\right )$ and 
$|S_2| \sim \operatorname{B} \left({2n/3-2K, 1/2 + 3j/n^{1/2}} \right)$ respectively. Therefore, using that $|S_1|$ and $|S_2|$ are independent random variables and applying Chernoff's inequality we have
\begin{align*}
{\mathbb P}(C_k \in [n]_{i,j}) & \geq 
{\mathbb P}(|S_1| \mbox { satisfies } \eqref{conditionforS_1} \mbox{ and } |S_2| \mbox { satisfies } \eqref{conditionforS_2})\\
& = 
{\mathbb P}(|S_1| \mbox { satisfies } \eqref{conditionforS_1}){\mathbb P}(|S_2| \mbox { satisfies } \eqref{conditionforS_2})\\
& \geq 
(1 - 2 e^{-\frac {2 (L/2)^2}{n/3-K}})(1 - 2 e^{-\frac {2 (L/2)^2}{2n/3-2K}})\\
& \geq
(1 - 2 e^{- \frac {n/2}{n/4}})(1 - 2 e^{- \frac {n/2}{n/2}})\\
& \geq 1/8.
\end{align*}
The third inequality above follows since $K \leq n/12$. This gives the first bound.

We now prove \eqref{comparingelementsinsamezone}. Suppose first that $A,B \in [n]_{i,j}$ with $r_B = r_A +1$ and $s_A = s_B$. Now 
\begin{align*}
 {\mathbb P}(C_k = A) 
 & = {\mathbb P}\Big (A\cap U = U_k \mbox{ and } S_1 = A \cap ([n/3]\setminus U)\Big )\\ 
 & \phantom{{}=1} \times {\mathbb P}\Big ( A \cap V = V_k \mbox { and }S_2 = A \cap ([n/3 + 1,n]\setminus V)\Big )\\ 
 & = \frac {\binom {n/6 + r_A}{k} \binom {{n/6} - r_A}{K-k}} {\binom {n/3}{K}} (p_{1,i})^{n/6 +r_A - k} (1-p_{1,i})^{n/6 - r_A - K + k}\\
& \phantom{{}=1} \times \frac {\binom {n/3 + s_A}{2K - 2k} \binom {n/3 - s_A}{2k}}{\binom {2n/3}{2K}}  (p_{2,j})^{n/3 + s_A - 2K + 2k} (1-p_{2,j})^{n/3 - s_A - 2k}.
\end{align*}
This gives
\begin{align*}
 \frac {{\mathbb P}(C_k = A)}{{\mathbb P}(C_k = B)} 
%
	  %& = \frac {\binom { { \frac {n}{6} } + r_A} {k} \binom { { \frac {n}{6} } - r_A } { K-k } (p_{1,i})^{n/6 +r_A - k} (1-p_{1,i})^{n/6 - r_A - K + k} }
%
	  %{\binom {{\frac {n}{6}} + r_A + 1}{k} \binom {{\frac {n}{6}} - r_A - 1}{K-k} (p_{1,i})^{n/6 +r_A + 1 - k} (1-p_{1,i})^{n/6 - r_A -1 - K + k}}\\
%
	  & =  \frac { ({\frac {n}{6}} + r_A - k + 1) (\frac {n}{6} - r_A) (1-p_{1,i}) }
	  { ({\frac {n}{6}} + r_A + 1    ) (\frac {n}{6} - r_A - K + k) p_{1,i}}\\
	  & = \Big (1 - \frac {k}{n/6 + r_A + 1} \Big) \Big( 1 + \frac {K - k}{n/6 - r_A - K + k} \Big ) \frac { (1 - 12i/n^{1/2}) } { (1 + 12i/n^{1/2}) }
\end{align*}
Applying the estimates (i) $1 + x \leq e^{2x}$ valid for all $x\in [0,1]$ and (ii) $e^{2x} \leq 1 + x$ for $x\in [-1/2 ,0]$ together with the bounds $12|i|/n^{1/2}\leq 1/2$, $|r_A| \leq (n\log n)^{1/2}$  and $k\leq K = n^{1/2}/12$, 
this gives
\begin{align}
\notag
 \frac { {\mathbb P} (C_k = A) } { {\mathbb P} (C_k = B) } 
						     & \leq  e^ { 2k/(n/6 + r_A + 1)} e^ { 2(K - k)/(n/6 - r_A - K + k) } e^ { 24|i|/ n^{1/2} } e^ { 24|i|/n^{1/2} }\\
\notag						     & \leq  e^ { (n^ {1/2}/6)/(n/6 - (n\log n)^ {1/2}) } . e ^{ (n^{1/2}/6)/( n/6 - (n\log n)^{1/2} - n^{1/2}/12) } \\
\notag						           & \phantom{{}=1}\times e^ { 12 (\log n/n)^{1/2} } e^ { 12 (\log n/n)^{1/2} }\\
\label{localchange1}				     & \leq e^ {30 (\log n/n)^{1/2}}.	
\end{align}	
A similar calculation shows that
\begin{equation}
\label{localchange2}
 \frac {{\mathbb P}(C_k = A)}{{\mathbb P}(C_k = B)} \geq e ^{-30(\log n/n)^{1/2}}.
\end{equation}
Furthermore, an identical argument gives that \eqref{localchange1} and \eqref{localchange2} hold if $A,B \in [n]_{i,j}$ with $r_A = r_B$ and $s_A = s_B +1$. Therefore given 
any two sets $A$ and $B$ in $[n]_{i,j}$, by repeatedly using \eqref{localchange1} and \eqref{localchange2} to change $r_A$ to $r_B$ and $s_A$ to $s_B$, we find 
that 
\begin{align*}
 \frac { {\mathbb P} (C_k = A) } { {\mathbb P} (C_k = B) } & \leq 
e^{ 30(\log n/n) ^{1/2} .( |r_A - r_B| + |s_A - s_B| ) }\\
& \leq 
e^ { 30(\log n/n)^{1/2}.(4L)} = e^ { 120 (\log n) ^{1/2} } = \alpha (n).
\end{align*}
Here we have used that $|r_A - r_B|\leq 2L$ and $|s_A - s_B| \leq 2L$ for all $A,B \in [n]_{i,j}$. This gives \eqref{comparingelementsinsamezone} and therefore concludes the lemma. \end{proof}

%
%
%
%
% Final section of chaper 
%
%
%

\section{Concluding remarks}

Let $G$ and ${\mathbf{x}}$ be as in the statement of Theorem \ref{genstickout}. Our original aim in this paper was to show that there exists a function $f: (0,1] \rightarrow (0,1]$ with $f(\alpha ) \rightarrow 0$ as $\alpha \rightarrow 0$ such that the following holds: for $n>n_0$ if $w(G) \leq \alpha 2^n$ then $|{\mathcal A}| \leq f(\alpha )2^n$ for all $(G,{\mathbf{x}})$-Sperner families ${\mathcal A}$. Thus the layered families control the size of all allowable families. Theorem \ref{genstickout} shows that this is true in a stronger form: we can actually take $f(\alpha ) = (1+o(1))\alpha$.
 
A natural question is the following: what happens if we replace the graph $G$ in Theorem \ref{genstickout} by a $3$-uniform hypergraph $H$ on vertex set $\{0,\ldots ,n\}$? Here for each edge $e=ijk$ of $H$ we would forbid a fixed `intersection pattern' $P_{ijk}$ between sets in $A\in [n]^{(i)}, B\in [n]^{(j)}$ and $C\in [n]^{(k)}$. This pattern would be described by the sizes of the intersections $A\cap B\cap C$, $A\cap B\cap C^c$,\ldots , $A^c\cap B^c\cap C^c$. Is it true that, as in Theorem \ref{genstickout}, the maximum size of a $(H, \mathbf{P})$-Sperner family (those families which do not contain one of these patterns) can again be controlled by $w(H)$ (where $w(H)$ is defined as before)? That is, does a function $f$ as above still exist for $3$-uniform hypergraphs? If this is true then it is easily seen that some restrictions on values of the $P_{ijk}$ are needed, like those on $x_{ij}$ in Theorem \ref{genstickout} -- for example, such patterns must satisfy $|A\cap B|,|A\cap C|,|B\cap C| \gg_{\alpha } n^{1/2}$, $|A\cap B \cap C| \gg_{\alpha } 1$ and $|A^c\cup B^c\cup C^c| \gg_{\alpha } 1$. 

%It would also be very interesting to extend Theorem \ref{genstickout} to the $k$-cube, $[k]^{n}$. 
%Given non-negative integers $a_1,\ldots ,a_k\in \mathbb{N}$ with $a_1+\ldots + a_k=n$, the 
%$(a_1,\ldots ,a_k)$-slice of $[k]^n$, denoted by $\binom {[n]}{a_1,\ldots ,a_k}$, consists 
%of all elements of $[k]^n$ with exactly $a_1$ coordinates with value $1$, $a_2$ coordinates 
%with value $2$, and so on. Slices in $[k]^n$ are the analogue of layers in $\mathcal {P}[n]=[2]^n$. 
%Do results analogous to Theorem \ref{genstickout} hold if we have a graph (or hypergraph) $G$ on the 
%set of slices of $[k]^n$? Here for each edge $e$ of $G$ we forbid a fixed pattern $P_e$ between 
%sets in slices that lie in $e$. For example, the density Hales-Jewett theorem could be 
%phrased in this language, for an appropriate hypergraph $H$ whose patterns correspond to 
%combinatorial lines. Can we still control the size of the maximal 
%$(G, \mathbf{P})$-Sperner families based on the size of $w(G)$?

Lastly, it would be interesting to know whether the upper bound in Theorem \ref{ordertilted} can be 
replaced by $|{\mathcal A}| \leq C\binom {n}{n/2}$, for some fixed constant $C>0$.

\section*{Acknowledgements}

I would like to thank Imre Leader for many helpful conversations and suggestions. I would also like to thank the anonymous referees for their careful reading of the paper and for several suggestions which have improved its structure.


\begin{thebibliography}{99}
\bibitem{aands} N. Alon and J. Spencer:
\textit{The Probabilistic Method}, Wiley, 3rd ed., 2008.
%
\bibitem{chernoff} H. Chernoff: A measure of asymptotic efficiency for tests of a hypothesis based on the sum of observations,
\textit{Ann. Math. Stat.} \textbf{23}(4) (1952), 493-507.
%
\bibitem{com} B. Bollob\'as: \textit{Combinatorics: Set Systems, Hypergraphs,
Families of Vectors and Combinatorial Probability}, Cambridge University Press,
1st ed., 1986.
%
\bibitem{eng} K. Engel: 
\textit{Sperner Theory}, Cambridge University Press, 1997.
%
\bibitem{forbid} P. Frankl, V. R\"odl: Forbidden intersections, 
\textit{Trans. Amer. Math. Soc.} \textbf{300} (1987), 259-286.
%
\bibitem{FK} H. Furstenberg and Y. Katznelson: A density version of the Hales-Jewett Theorem, 
\textit{Journal d’Analyse Math´ematique}, \textbf{57} (1991), 64-119.
%
\bibitem{GRS} D. Gunderson, V. R\"odl, A. Sidorenko: Extremal problems for sets forming Boolean 
algebras and complete partite hypergraphs, 
\textit{J. Comb. Theory, Ser. A} \textbf{88}(2) (1999), 342-367 .
%
\bibitem{kal} G. Kalai: Personal communication (2010).
%
\bibitem{katona} G.O.H. Katona, A simple proof of the Erd\" os -- Chao Ko -- Rado
theorem, \textit{J.\ Comb.\ Theory (B)} \textbf{13}(1973), 83-84.
%
\bibitem{kl} P. Keevash, E. Long: Frankl-R\"odl type theorems for codes and permutations, \emph{submitted}. \href{http://arxiv.org/abs/1402.6294}{arXiv:1402.6294} 
%
\bibitem{LL} I. Leader, E. Long: Tilted Sperner families, 
\textit{Disc. Appl. Math.} \textbf{163}(2) (2014), 194-198.
%
\bibitem{DHJ} D.H.J. Polymath: A new proof of the density Hales-Jewett theorem, 
\textit{Ann. of Math.} \textbf{175} (2012) 1283-1327.
%
\bibitem{sper} E. Sperner: Ein Satz \"uber Untermengen einer endlichen Menge. \textit{Math. Z.}, \textbf{27}(1928), 544-548.

\end{thebibliography}
\end{document}